\documentclass[11pt]{amsart}
\usepackage{geometry,graphicx}

\usepackage{latexsym}
\usepackage{labelfig}
\usepackage{amssymb}
\usepackage{enumerate}
\usepackage[cp850]{inputenc}
\usepackage{epsfig}
\usepackage{epstopdf}
\usepackage{psfrag}
\usepackage{amsthm}
\usepackage{amscd}
\usepackage{amsmath}
\usepackage{amsfonts}
\usepackage{graphics}
\usepackage{enumerate}
\usepackage{color}
\theoremstyle{theorem}

\newtheorem{theorem}{Theorem}

\newtheorem*{maintheorem}{Theorem}

\newtheorem{proposition}{Proposition}

\newtheorem{lemma}{Lemma}

\theoremstyle{definition}

\newtheorem{remark}{Remark}

\newcommand{\sys}{{\rm sys}}

\newcommand{\R}{{\mathbb R}}

\newcommand{\Hyp}{{\mathbb H}}

\newcommand{\N}{{\mathbb N}}

\newcommand{\Z}{{\mathbb Z}}

\newcommand{\PSL}{{\rm PSL}}

\newcommand{\arcsinh}{{\,\rm arcsinh}}

\newcommand{\arccosh}{{\,\rm arccosh}}

\newcommand{\length}{{\rm length}}

\numberwithin{equation}{section}

\newcommand{\sfrac}{\frac}

\begin{document}

%




%

\title[Systole growth on surfaces with cusps]{Systole growth for finite area hyperbolic surfaces}




\author[F.~Balacheff]{Florent Balacheff}

\address[Florent Balacheff] {Laboratoire Paul Painlev\'e, Universit\'e Lille 1\\ Lille, France}
\email{florent.balacheff@math.univ-lille1.fr}
\author[E. Makover]{Eran Makover}
\address[Eran Makover]{Central Connecticut State University\\ Department of Mathematics \\New Britain, CT}
\email{makovere@ccsu.edu}

\author[H.~Parlier]{Hugo Parlier${}^\dagger$}

\address[Hugo Parlier]{Department of Mathematics, University of Fribourg\\
  Switzerland}
\email{hugo.parlier@gmail.com}
\thanks{${}^\dagger$Research supported by Swiss National Science Foundation grant number PP00P2\textunderscore 128557}

\date{\today}




%

\begin{abstract} 
We are interested in the maximum value achieved by the systole function over all complete finite area hyperbolic surfaces of a given signature $(g,n)$. This maximum is shown to be strictly increasing in terms of the number of cusps for small values of $n$. We also show that this function is greater than a function that grows logarithmically in function of the ratio $g/n$.
\end{abstract}




%



\maketitle

\section{Introduction}

Consider two natural integers $g$ and $n$ such that $2g-2+n$ is positive. For an orientable complete finite area hyperbolic surface $M$ of signature $(g,n)$, where $g$ is the genus and $n$ the number of cusps, the {\it systole} is defined as length of the (or a) shortest closed geodesic and denoted by $\sys(M)$. Because there is an upper bound on the systole which only depends on the topology of the surface, the quantity
$$
\sys(g,n):=\sup \{\,\sys(M) \mid M \, \text{ is a finite area hyperbolic surface of signature} \, (g,n)\,\}
$$
is finite and we are interested in its asymptotic behavior in terms of $g$ and $n$. Observe that this supremum is in fact a maximum following Mahler's compactness theorem, see \cite{Mu71}. Furthermore $\sys(g,n)$ is universally bounded from below by $2 \arcsinh(1)$, an easy consequence of, on the one hand, the fact that the set of systoles of a maximal surface fill the surface, and on the other hand, the collar lemma.

The first interesting upper bound on $\sys(g,n)$ to appear in the literature is due to Schmutz Schaller \cite[Theorem 14]{Sc94} (another proof appeared in \cite{Ad98}). The bound is as follows: for $n\geq 2$ and $(g,n)\neq (0,3)$,
\begin{equation}\label{eq:schmutz}
\sys(g,n) \leq 4 \arccosh \left({6g-6+3n\over n}\right).
\end{equation}
For a number of cusps bigger than the genus, the function $\sys(g,n)$ is thus roughly constant. More precisely, for any function $n(g)$ of the genus with integer values such that $\sup_{g\to \infty} \sfrac{g}{n(g)}=\alpha \in [0,\infty[$, we have
$$
2 \arcsinh(1)\leq \sys(g,n(g)) \leq 4 \arccosh(6\alpha+3).
$$ 
In the other direction, if $n(g)$ is such that $\lim_{g\to \infty} \sfrac{g}{n(g)}=+\infty$, the right term in inequality (\ref{eq:schmutz})  is asymptotically equivalent to $4\ln (\sfrac{g}{n(g)})$. 

In \cite{Sc94}, Schmutz Schaller also proved that surfaces corresponding to principal congruence subgroups of $\PSL_2(\Z)$ have maximal systole in their respective moduli spaces. More precisely, for any integer $k\geq 2$, let $\Gamma_k$ denote the kernel of the map $\PSL_2(\Z)\to \PSL_2(\Z/k\Z)$. If $(g_k,n_k)$ denotes the signature of the surface quotient $M_k=\Hyp^2/\Gamma_k$, then
$$
\sys(M_k)=\sys(g_k,n_k).
$$
For $k=p$ prime, 
$$
n_p=\frac{p^2-1}{2}
$$ 
and 
$$
g_p=1+\frac{(p^2-1)(p-6)}{24}
$$
so that  $72\cdot g_p^2 \simeq n_p^3$ (see \cite{BFK}). Furthermore, a standard calculation (compare with \cite{BS94}) leads to the following lower bound:
\begin{equation}\label{eq:arithm}
\sys(M_p)\gtrsim 4 \ln \left({g_p\over n_p}\right) \simeq {4\over 3}\ln g_p.
\end{equation}
So inequality (\ref{eq:schmutz}) is asymptotically sharp for signatures $(g_p,n_p)$. Here and in the sequel, we use the following definition: two functions $\lambda, \mu : \N \to \R$ satisfy the relation $\lambda(g) \gtrsim \mu(g)$ if for any positive $\varepsilon$ we have $(1+\varepsilon) \cdot \lambda(g) \geq \mu(g)$ for large enough $g$. \\

In the compact case, Buser and Sarnak proved in \cite{BS94} that there exists a universal constant $U>0$ such that for all genus $g\geq 2$
$$
U\ln g \leq \sys(g,0) \leq 2 \ln (4g-2),
$$
and that for some special infinite sequences of genera $g_k$, the lower bound can be strengthened to
$$
\sys(g_k,0)\gtrsim {4\over 3} \ln g_k
$$ 
by more delicate arithmetic considerations.\\

We add the following results to this panorama.

\begin{maintheorem}\label{thm:main} The function $\sys(g,n)$ enjoys the following properties.\\

\noindent 1. For all $(g,n)$ such that $2g-2 + n >0$ and $n\leq 2$,
$$
\sys(g,n)< \sys(g,n+1).
$$
\noindent 2. For any $g\geq 2$ and $n \in \N$, 
$$
\sys(g,n)\geq U \ln \left({g \over n+1}\right)
$$  
where $U$ denotes  Buser's and Sarnak's universal constant.\\

\end{maintheorem}

Property 1. says that for fixed genus the function $\sys(g,n)$ is strictly increasing in $n$ for very small values of $n$. One may wonder how long this remains true. The existence of closed hyperbolic surfaces in any genus $g\geq 2$ with large systole \cite{BS94} and inequality (\ref{eq:schmutz}) imply that this growth property is doomed to fail for large enough $n$. It is worth noting that it is unknown whether $\sys(g,n)$ is increasing in $g$.

Property 2. together with inequality (\ref{eq:schmutz}) implies that for {\it any} function $n:\N \to \N$ such that $\lim_{g\to \infty} {\sfrac{g}{n(g)}}=+\infty$, the function $\sys(g,n(g))$ grows {\it roughly} like $\ln \left({\sfrac{g}{n(g)}}\right)$ (where by roughly we mean up to positive multiplicative constants).

\section{Adding a small number of cusps}

Our goal in this section is to prove Property 1. of our main theorem. This follows directly from the following  two lemmas.

\begin{lemma} Let $M$ be a hyperbolic surface of signature $(g,n)$ and suppose that $M$ admits an embedded open disk of radius $\sfrac{\sys(M)}{2}$ centered at some point  $p$. Then the unique hyperbolic surface $M'$ of signature $(g,n+1)$   and conformally equivalent to $M\setminus \{p\}$ satisfies
$$
\sys(M')>\sys(M).
$$
\end{lemma}

\begin{proof}
In order to link the systole on both surfaces, we first define an appropriate notion of systole for Riemannian surfaces of genus $g$ with $n$ puncture points. For such a surface, the systole is defined as the infimum over the lengths of all closed curves which does not bound a disk or a once punctured disk. For hyperbolic closed surfaces of genus $g$ or hyperbolic surfaces of signature $(g,n)$, this definition coincides with the definition of systole as the length of a shortest closed geodesic. 

We now control the systole of $M\setminus \{p\}$ in terms of the systole of $M$. On the one hand, a closed curve of $M\setminus \{p\}$ which does not bound a disk or a punctured disk in $M$ has length at least $\sys(M)$ by definition. On the other hand, let $\gamma$ be a closed curve of $M\setminus\{p\}$ which bounds a once punctured disk in $M$ but not in $M\setminus\{p\}$. If we consider all closed curves of $M\setminus\{p\}$ homotopic to $\gamma$, the infimum of their lengths will be realized by a closed geodesic loop of $M$ based at $p$. As the open disk centered at $p$ and of radius $\sfrac{\sys(M)}{2}$ is embedded, we deduce that the length of this geodesic loop, and thus of $\gamma$,  is at least $\sys(M)$. In conclusion,  
$$
\sys(M\setminus \{p\})=\sys(M).
$$
The result then follows from a type of Schwarz lemma \cite{Wo82} which asserts that  the natural map $M' \to M$ is contracting.
\end{proof}

\begin{lemma}
On a hyperbolic surface $M$ of genus $g$ with at most two cusps, there always exists an embedded open disk of radius  $\sfrac{\sys(M)}{2}$. 
\end{lemma}

\begin{proof}
If there is no cusp, this is a straightforward consequence of the definition of systole. 

So suppose that there is at least one cusp on the surface and proceed by contradiction. Consider a maximal embedded disk $D$ and denote by $p$ its center.
The existence of such a disk is guaranteed by compactness (the thick part of the surface is compact) and the fact that there is an upper bound on the radius of an embedded disk. By assumption, its radius $r$ is strictly less than $\frac{\sys(M)}{2}$. Due to the maximality of $D$, its boundary admits at least one point of self-intersection which we call a {\it self-bumping point} of the disk. Consider the geodesic loop formed by the two radii of the disk joining $p$ to some self-bumping point. As its length is $2r< \sys(M)$, this {\it bumping loop} must be parallel to a cusp. If there is only one such bumping loop, it is easy to see that we get a contradiction: by moving our point $p$ in the appropriate direction, one can increase the radius of our maximal embedded open disk. So there are at least two bumping loops. Now observe that because the disk is embedded, two such loops only intersect in $p$ and thus both loops are non homotopic (see for instance the bigon criterion \cite{FM11}). So both loops surround distinct cusps. In particular for the case of a surface with one cusp we have reached a contradiction. We now suppose for the sequel that $M$ has two cusps. 

Observe that the existence of a third self-bumping point is impossible, as there are no more cusps to surround. We will get a contradiction in the event that there are only two self-bumping loops on the maximal embedded disk by proving that we can move its center in an appropriate direction in order to increase its radius. Recall that two horocyclic neighborhoods of area $1$ around the cusps are disjoint. Now consider the two geodesic loops $\gamma_1,\gamma_2$ based at $p$ and surrounding the two cusps. Consider, for each of the neighborhoods, the unique shortest geodesic between the neighborhood and the point $p$ entirely contained in the cylinder bounded by the cusp and the corresponding geodesic loop (see figure \ref{fig:cusp}).

\begin{figure}[h]
\leavevmode \SetLabels
\endSetLabels
\begin{center}
\AffixLabels{\centerline{\epsfig{file =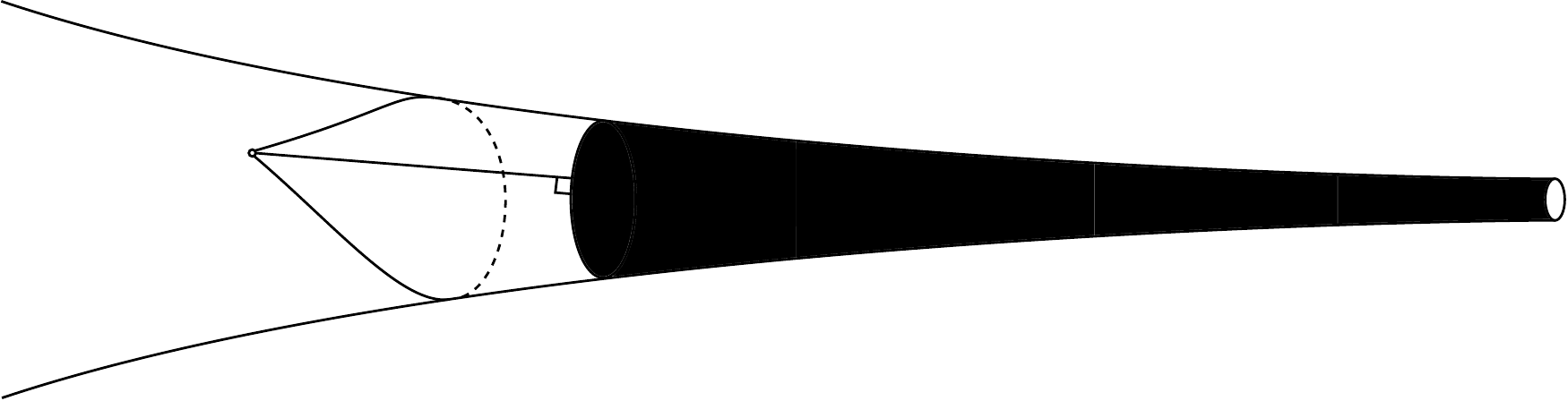,width=12.0cm,angle=0}}}
\vspace{-18pt}
\end{center}
\caption{A geodesic loop surrounding a cusp} \label{fig:cusp}
\end{figure}

For each  loop $\gamma_k$, denote this distance $d_k$ and its associated path $c_k$. Via standard hyperbolic trigonometry considerations, the geometry of this hyperbolic cylinder between one of the geodesic loops and the associated cusp  is completely determined by either $\ell(\gamma_k)$ or $d_k$. More precisely, there exists a strictly increasing function that associates to the length $\ell(\gamma_k)$ the distance $d_k$. In particular, as $\ell(\gamma_1)=\ell(\gamma_2)$, we deduce that $d_1=d_2$. Now consider the two distance realizing paths $c_1$ and $c_2$ emanating from $p$. They locally divide the surface around $p$ into two parts, and as such one can move the point $p$ in the appropriate part in order to increase its distance to each of the horocyclic neighborhoods. This leads to a contradiction.
\end{proof}

\section{Adding a large number of cusps}

The purpose of this section is to prove Property 2. of our main theorem.

\subsection{Comparison results and consequences}
Fix a signature $(g,n)$ such that $g\geq 2$. Let $M$ be a hyperbolic closed surface of genus $g$ with $n$ points $\{p_1,\ldots,p_n\}$ and $M'$ be a hyperbolic punctured surface with $n$ cusps. Suppose that $M'$ and $M\setminus \{p_1,\ldots,p_n\}$ are conformally equivalent. We already know by \cite{Wo82} that the natural map $M' \to M$ is contracting, and thus $\sys(M')>\sys(M)$, but we would like to have an opposite type inequality.

We first recall a comparison theorem of Brooks found in \cite{Br99}.  We denote by $D_i:=D(p_i,r)$ the disks in $M$ of radius $r$ centered at the $p_i$'s. For small enough $r$, these disks are pairwise disjoint and embedded. We also denote by  $\{c_1,\ldots,c_n\}$ the $n$ cusps of $M'$. Each cusp $c_i$ admits a neighborhood isometric to a neighborhood of infinity in the punctured disk model ${\mathcal C}:={\mathbb H}^2 / \{z \mapsto z+1\}$. More specifically, we can choose  around each cusp $c_i$ a horocyclic neighborhood of size $\ell$ denoted $B_i:=B(c_i,\ell)$ and composed of all horocycles around $c_i$ of length at most $\ell$. Thus $B_i$ is isometric to the subset $\{z \mid \Im(z)\geq {1 \over \ell}\}$ of ${\mathcal C}$ where $\Im(z)$ denotes the imaginary part of $z$. By choosing $\ell$ small enough, we can assume that these neighborhoods are pairwise disjoint. 

\begin{theorem}[Brooks]\label{th:brooks}
For any $\varepsilon>0$, there exist numbers $r$ and $\ell$ such that, if the disks $D_1,\ldots,D_n$ of radius $r$ and the horocyclic neighborhoods $B_1,\ldots,B_n$ of size $\ell$ are respectively pairwise disjoint, then to every closed geodesic $\gamma$ in $M$ corresponds a closed geodesic $\gamma'$ in $M'$ whose image under the natural map $M' \to M$ is homotopic to $\gamma$ and such that
$$
\length(\gamma')\leq(1+\varepsilon) \cdot\\length(\gamma).
$$
In particular,
$$
\sys(M')\leq(1+\varepsilon) \cdot\sys(M).
$$
\end{theorem}

As the surfaces $M_p=\Hyp^2/\Gamma_p$ for $p$ prime corresponding to the $p$-th congruence subgroup satisfy
$$
\sys(M_p)\gtrsim {4 \over 3} \ln g_p,
$$
we directly deduce that 
$$
\sys(g_p,0) \gtrsim {4 \over 3} \ln g_p.
$$ 
 \\

We proceed to a more precise lower bound on the systole of $M'$ in terms of the systole of $M$ and the radius $r$.  

\begin{proposition}\label{prop:sysestimate}
For any $r$ such that the disks $D_1,\ldots,D_n$ of radius $r$ are pairwise disjoint, 
$$
\sys(M')> \min \{\sys(M), 4r\}.
$$
\end{proposition}

\begin{proof}
Any closed geodesic of $M'$ whose image under the map $M' \to M$ is not homotopically trivial has length strictly bounded from below by $\sys(M)$. We will prove that the remaining homotopy classes have geodesic representatives of length greater than $4r$. Among such classes, a class whose geodesic representative is of minimal length is necessarily simple. Consider such a class in $M'$. This class defines a unique homotopy class in $M\setminus\{p_1,\ldots,p_n\}$ which is uniquely represented by a minimal piecewise geodesic in $M$ where the singular points of the curve belong to $\{p_1,\ldots,p_n\}$. The number of these singular points is at least two and so, as two points of the $p_i$'s are distance at least $2r$, its length is at least $4r$. To conclude the proof, we again use that the natural map $M' \to M$ is contracting \cite{Wo82}.
\end{proof}

Property 2. is now implied by the following.

\begin{proposition}\label{prop:log}
For any $n \in \N^\ast$, 
$$
\sys(g,n)\geq \min \left\{U \ln g, 2\arccosh \left({2(g-1) \over n}+1\right)\right\}
$$  
where $U$ denotes the Buser's and Sarnak's universal constant. 
\end{proposition}

\begin{proof}
We know by \cite{BS94} that 
$$
\sys(g,0)\geq U \, \ln g
$$
for any $g\geq 2$ where $U>0$ is universal constant. Let $M$ be a maximal surface of genus $g$ and consider a maximal system $\{D_i\}_{i=1}^N$ of pairwise disjoint disks of radius 
$$
r:=\min \left\{{U\over 4} \ln g,{1\over 2}\arccosh \left({2(g-1) \over n}+1\right)\right\}.
$$ 
Observe that $r\leq \sfrac{\sys(M)}{4}$ and denote by $p_i$ the center of the disk $D_i$. By the maximality of such a system, $M$ must be covered by the disks $\{2D_i\}_{i=1}^N$, where $2D_i$ denotes the disk of center $p_i$ and radius $2r$. We derive that
$$
N \cdot 2\pi (\cosh 2r-1)\geq N \cdot A(2D_i) \geq A(M)=4\pi (g-1),
$$
so
$$
N \geq n.
$$
Now let $M'$ be the unique surface with $n$ cusps such that $M \setminus \{p_1,\ldots,p_n\}$ and $M'$ are conformally equivalent. We conclude by applying Proposition \ref{prop:sysestimate}.
\end{proof}

\begin{remark}\label{rem:inf}
In particular for any $\alpha \in [0,1[$, 
$$
\sys(g,[g^\alpha])\geq c_\alpha \, \ln g,
$$  
where $c_\alpha:=\min \{U,2(1-\alpha)\}$. 
\end{remark}

\subsection{Final remarks}

Fix a growing function $n:\N \to \N$ and let $M'$ be a maximal surface of genus $g$ with $n(g)$ cusps, that is
$$
\sys(M')=\sys(g,n(g)).
$$
Theorem \ref{th:brooks} applied to large enough $g$ implies the following:
$$
\sys(g,0)\gtrsim \sys(g,n(g)).
$$
It would be interesting to know for what $n(g)$ the opposite inequality holds. In particular, observe that proving that $\sys(g,n(g))\simeq \sys(g,0)$ for functions satisfying $n(g)\leq 5[g^{\sfrac{2}{3}}]$ would imply the conjecture \cite{Sc98} that 
$$
\sys(g_p,0)\simeq {4\over3}\ln g_p.
$$








\begin{thebibliography}{99}


\bibitem[Ad98]{Ad98}
Adams, C.: {\it Maximal cusps, collars , and systoles in hyperbolic surfaces}. Indiana Math. J. {\bf 47} (1998), no. 2, 419-437.


\bibitem[Br99]{Br99}
Brooks, R.: {\it Platonic surfaces}. 
Comment. Math. Helv. {\bf 74} (1999), no. 1, 156-170.

\bibitem[BFK]{BFK}
Brooks, R., Farkas, H. and Kra, I.: {\it Number theory, theta identities, and modular curves. Extremal Riemann surfaces (San Francisco, CA, 1995)}. Contemp. Math. {\bf 201} (1997), 125-154, Amer. Math. Soc., Providence, RI. 

\bibitem[BS94]{BS94}
Buser, P. and Sarnak, P.: {\it On the period matrix of a Riemann surface of large genus. With an appendix by J. H. Conway and N. J. A. Sloane.} 
Invent. Math. {\bf 117} (1994), no. 1, 27-56.

\bibitem[FM11]{FM11} Farb, B. and Margalit, D.: {\it A primer on mapping class groups}. To appear in {\it Princeton Mathematical Series}.

\bibitem[Mu71]{Mu71}
Mumford, D.: {\it A remark on a Mahler's compactness theorem}. Proc. AMS {\bf 28} (1971), no. 1, 289-294.

\bibitem[Sc94]{Sc94}
Schmutz, P.: {\it Congruence subgroups and maximal Riemann surfaces}. J. Geom. Anal. {\bf 4} (1994), 207-218.

\bibitem[Sc98]{Sc98}
Schmutz Schaller, P.: {\it Geometry of Riemann surfaces based on closed geodesics}. 
Bull. Amer. Math. Soc. (N.S.) 35 (1998), no. 3, 193ð214.

\bibitem[Wo82]{Wo82}
Wolpert, S.: {\it A generalization of the {A}hlfors-{S}chwarz lemma}.
Proc. Amer. Math. Soc. {\bf 84} (1982), no. 3, 377-378.


\end{thebibliography}
\end{document}